%% file: ams-pos.tex
\documentclass[%
	,a4paper 
	,oneside
	,english
	,10pt
	]{amsart}
\input{my-ams} 
\usepackage{lmodern}
\usepackage{ragged2e}
\usepackage[square,numbers]{natbib}
\usepackage{doi}
\title[Spectral Properties of 
Positive Semigroups]{Spectral Properties of 
$ C_{{0}}$-Semigroups of Positive Operators on \CA-Algebras} 
\date{\today}
\author{Ulrich Groh}
\address{%
Mathematisches Institut \\
Universit\"at T\"ubingen \\
	Auf der Morgenstelle 10 \\
	72076 T\"ubingen 
	(Germany)
	}
\email{ulgr@math.uni-tuebingen.de}
\begin{document}
\thispagestyle{empty}
\begin{abstract}
Let $ (T(t))_{t\geq0} $ be a positive $C_{0}$-semigroup with generator $A$ on a \CA-algebra or on the predual of a \WA-algebra.
Then the growth bound $\omega_{0}$ equals $\sA$ and if $\Sp(A) \neq \emptyset$, then $\sA$, the spectral bound of $A$, is a spectral value.
\end{abstract}
\maketitle
\input{ams-sec1}	
\input{ams-sec2}	
\input{ams-sec3}	
{\RaggedRight
\bibliographystyle{abbrvnat}
\bibliography{ams-pos}
}
\end{document}

%% file: my-ams.tex
\usepackage{amssymb}
\let\phi\varphi
\let\epsilon\varepsilon
\let\leq\leqslant
\let\geq\geqslant
\usepackage{parskip}
\usepackage[english]{babel}
\usepackage[english=american]{csquotes}
\usepackage[english,backgroundcolor=white]{todonotes}

\usepackage[inline,shortlabels]{enumitem}
\setlist{\upshape (i), wide, labelindent=.5em, itemsep=.25em}
\theoremstyle{plain}
\newtheorem{theorem}{Theorem}[section]
\newtheorem{proposition}[theorem]{Proposition}
\newtheorem{lemma}[theorem]{Lemma}
\newtheorem{corollary}[theorem]{Corollary}
\theoremstyle{definition}

\theoremstyle{remark}

\newtheorem{remarks}[theorem]{Remarks}

\usepackage{mathtools}
\newcommand{\rointerval}[1]{\left[ #1 \right[}	
\DeclarePairedDelimiterX{\norm}[1]{\lVert}{\rVert}{\ifblank{#1}{\:\cdot\:}{#1}}
\DeclarePairedDelimiterX{\abs}[1]{\lvert}{\rvert}{\ifblank{#1}{\:\cdot\:}{#1}}
\usepackage{xspace}
\AtBeginDocument{%
	
	\newcommand{\ie}{i.e.,\xspace}
	
	\renewcommand{\L}[1]{\mathcal{L}(#1)}} 	

\newcommand*{\ds}{\mathop{}\mathrm{d}{s}}	
\newcommand*{\iu}{\mathrm{i}}	
\newcommand*{\eu}{\mathrm{e}}	
\let\Re\relax
\DeclareMathOperator{\Re}{Re}	
\newcommand{\N}{\mathbb{N}}		
\newcommand{\R}{\mathbb{R}}		
\newcommand{\C}{\mathbb{C}}		
\newcommand{\Sp}{\sigma}	
\newcommand{\sA}{\mathrm{s}(A)} 
\newcommand{\Res}{\rho}	
\newcommand{\RR}{\mathcal{R}}
\newcommand{\CA}{C$^{\star}$}	
\newcommand{\WA}{W$^{\star}$}	
\newcommand{\BA}{\mathfrak{A}}	
\newcommand{\SA}{\mathbf{S}(\mathfrak{A})}	
\usepackage{bbm}
\newcommand{\1}{\mathbbm{1}}
\usepackage{hyperref}   
\hypersetup{
	,breaklinks = true
	,colorlinks	= true    	%
	,urlcolor	= blue    	%
	,citecolor	= blue    	%
	,linkcolor	= blue		%
}

%% file: ams-sec1.tex

\section{Introduction}
This short note arose in parallel with the preparation of a new edition of our lecture notes \enquote{One-parameter Semigroups of Positive Operators} \cite{nagel-ln:2025}, originally published by Springer in 1986.
Chapters D-I through D-IV, written in 1983, contain foundational results on the spectrum of generators of positive $C_{0}$-semigroups on \CA-algebras and the preduals of \WA-algebras.
Here we provide a more detailed treatment of the \enquote{$\sA = \omega_{0}$} problem, mentioned in the chapters but not extensively discussed.

One can clarify this problem using the theory of ordered Banach spaces: 
By \cite[Theorem 5.3.1]{zbmath03736445} one has $s(A) \in \sigma(A)$ whenever $A$ generates a positive semigroup on an ordered Banach space with normal cone and if the spectrum of $A$ is non-empty. 
Moreover, $s(A) = \omega_{0}$ if the norm is additive on the dual positive cone, see \cite[Theorem 2.4.4]{zbmath03883065}.

\CA-algebras enjoy both properties.
However, it is helpful to provide a summary in the context of \CA-algebras and preduals of \WA-algebras, as this allows for a more direct treatment of this topic.
\subsection*{What is the question?}
If $(T(t))_{t\geq0}$ is a strongly continuous operator semigroup on a Banach space, or, for short, a $C_{0}$-semigroup, with generator $A$, then we are interested in the relationship between properties of the spectrum $\Sp(A)$ of the generator and the convergence behavior of the semigroup $(T(t))_{t\geq0}$ (see \cite[A-IV \& B-IV]{nagel-ln:2025}).

To this end, two essential definitions are required (see \cite[Chap. A-III, 1. Introduction (1.2)--(1.3) ]{nagel-ln:2025}).
\begin{enumerate}
\item
First, the \emph{spectral bound} $\sA$, \ie
\[
	\sA = \sup \{ \Re(\lambda) \colon \lambda \in \Sp(A) \}
\]
where $ \Re(\cdot)$ denotes the real part of a complex number.
Then we always have
\[
	- \infty \leq \sA < \infty
\]
with $- \infty = \sA$ when the spectrum of $A$ is empty.
\item
The other quantity is the \emph{growth bound} $\omega_{0}$ which yields
\[
	r(T(t)) = \eu^{\omega_{0} t} \quad (t \geq 0).
\]
where $r(T(t))$ denotes the spectral radius of the operator $T(t)$.
\end{enumerate}
Therefore,  for $ \lambda \in \C$ with $ \Re(\lambda) > \omega_{0}$, the resolvent $\RR(\lambda,A)$ of $A$ exists, is given by
\[
	\RR( \lambda , A ) x = \int_{0}^{\infty} \eu^{ - \lambda s} T(s) x \ds
\]
and 
\[
	- \infty \leq \sA \leq \omega_{0} < \infty ,
\]
where all variations are possible (for this see Remarks~\ref{ex:remarks} on page~\pageref{ex:remarks}).
In the case that $ \sA = \omega_{0}$, we say that the spectrum of $A$ determines the asymptotic behavior of the semigroup $(T(t))_{t\geq0}$.
In general, one cannot expect this equality to hold (see \cite[A-III, Example 1.3]{nagel-ln:2025}) not even for positive $C_{0}$-semigroups.
\subsection*{Terminology}
We use terminology and notation from \emph{One-parameter Semigroups of Positive Operators} \cite{nagel-ln:2025}, Chapter A.
Results on positive semigroups on commutative \CA-algebras can be found in \cite[B-II \& B-III]{nagel-ln:2025}.
Our main reference for the theory of \CA-algebras is \cite{pedersen:2018}, for \WA-algebras \cite{sakai:1971}, and for positive operators on operator algebras \cite{stoermer:2013}.

%% file: ams-sec2.tex

\section{Positive Operators}%
If $T$ is an operator on a Banach space, we denote by $\Sp(T)$ its spectrum, by $r(T)$ the spectral radius and by $\Res(T)$ its resolvent set. 
If $\lambda \in \Res(T)$ with $\abs{\lambda} > r(T)$ then its resolvent $\RR(\lambda,T)$ is given by 
$\RR(\lambda,T) = \sum_{n=0}^{\infty} \lambda^{-(n+1)}T^{n}$
where the sum exists in the operator norm.
For these and further properties of the spectrum see \cite[Chap. VII]{conway:1990}.
In this section we show the fundamental property of the spectrum of positive operators on \CA-algebras needed later, \ie $r(T) \in \Sp(T)$.
\begin{lemma}
Let $T$ be a positive operator on a \CA-algebra $\BA$, let $\lambda \in \Res(T)$ such that $\abs{\lambda} > r(T)$.
Then for all $ 0 \leq \phi \in \BA^{*}$ and $ 0 \leq x \in \BA$, we have
\begin{equation}\label{eq:resolvent-inequality}
		\abs*{ \phi(\RR( \lambda, T )x) }
		  \leq \phi(\RR( \abs{\lambda}, T)x )  .
\end{equation}
\end{lemma}
\begin{proof}
Using representation of the resolvent $\RR( \lambda, T)$ of $T$ we obtain
\begin{equation*}
\begin{split}
	\abs*{\phi( \RR( \lambda, T) x ) }   
		&= \abs*{ \sum_{ k=0 }^{\infty} \, \lambda^{-(k + 1)} \phi(T^{ k }x)} \\
	\leq  \sum_{ k=0 }^{\infty} \abs*{\lambda}^{-(k + 1)} \phi(T^{ k }x) 
	 		&= \phi \left( \sum_{k=0}^{\infty} \, \abs{\lambda}^{-(k + 1)} \, T^{ k } x)\right) 
		= \phi( \RR( \abs{\lambda},T) x ).
\end{split}
\end{equation*}
\phantom{x}
\end{proof}
\begin{theorem}\label{thm:main-op}
Let $T$ be a positive operator on a \CA-algebra $\BA$.
\begin{enumerate}
\item
The spectral radius $ r(T) $ is an element of the spectrum $\Sp(T)$ of\/ $T$.
	
\item 
If\/ $\BA$ contains a unit $\1$, then there exists a 
$0 \leq \phi \in \BA^{*}$, $\norm{\phi} = 1$ 
such that $T'\phi = r(T) \phi$ where $T'$ denotes the adjoint operator of $T$ on the dual space $\BA^{*}$ of\/ $\BA$.
\end{enumerate}
\end{theorem}
\begin{proof}
{\upshape (i)}\quad
Take  $\alpha \in \Sp(T)$ with $\abs{\alpha} = r(T)$ and a sequence $(\lambda_{n})$ in the resolvent set $\Res(T)$ such that $\abs{\lambda_{n}} \downarrow  r(T)$ and $\lim_{n} \lambda_{n} = \alpha$.
Then, since $\alpha$ is a singularity of the resolvent of $T$, we conclude  
\[
	\limsup_{n} \norm*{ \RR( \lambda_{n}, T) } = \infty .
\]
Since the positive cones of $\BA$ and $\BA{*}$ are generating, there exists, by the uniform boundedness principle, $ 0 \leq \phi \in \BA^{*}$ and $0 \leq x \in \BA$, both of norm $\leq 1$ such that
\[
	 \limsup_{n} \abs*{ \phi (\RR( \lambda_{n}, T) x ) } = \infty  ,  
\]
and by Equation~\eqref{eq:resolvent-inequality}, 
\[
	\limsup_{n} \norm*{ \RR( \abs{\lambda_{n}}, T) )} = \infty .
\]
This proves $r(T) \in \Sp(T)$.

{\upshape (ii)}\quad
Denote $ r \coloneqq r(T) $ and let $ (\lambda_{n}) $ be a real sequence $\lambda_{n} \downarrow r$ in $\Res(T)$ in such that
\[
	\limsup_{n} \norm*{\RR( \lambda_{n}, T)} = \infty   .
\]
Again by the uniform boundedness principle and because the positive cone of $\BA^{*}$ is generating, there exists a positive linear form $\phi$ on $\BA$ such that 
\[
    \limsup_{n} \norm{ \RR( \lambda_{n} , T') \phi } = \infty .
\]
Since the resolvents $ \RR( \lambda_{n} , T') =  \RR( \lambda_{n} , T)' $ are positive operators on $\BA^{*}$, the linear forms
\[
    \phi_{n} = \norm{ \RR( \lambda_{n} , T') \phi }^{ -1 } \, 
    	\RR( \lambda_{n} , T') \phi \, ,
\]
are positive with $ \norm{\phi_{n}} = \phi_{n}(\1) = 1 $, hence are elements of the state space $\SA$ of $\BA$.

Since the state space $\SA$ is compact for the $\sigma(\BA^{*}, \BA)$-topology, this sequence has an accumulation point $ \phi_{0} $ with $ \phi_{0}(\1)= 1$. 
By compactness we find an ultrafilter $ \mathfrak{U} $ on $\N$ such that
\[
    \lim_{\mathfrak{U}} \phi_{n} = \phi_{0} .
\]
Since 
\begin{equation}\label{eq:convergence}
\begin{split}
	 (r - T')\phi_{n} & = \norm{ \RR( \lambda_{n} , T') \phi }^{ -1 }
	 	[ ( r - \lambda_{n}) + (\lambda_{n} - T) ]
	 		\RR( \lambda_{n} , T') \phi \\
	 	& = (r - \lambda_{n})\phi_{n} + \norm{ \RR( \lambda_{n} , T')	             
			\phi }^{ -1 } \phi 
\end{split}
\end{equation}
and since the adjoint operator $T'$ is continuous for the $\sigma(\BA^{*}, \BA)$-topology on $\BA^{*}$, we obtain from Equation~\eqref{eq:convergence} above that
\[
    (r - T') \phi_{0} = \lim_{ \mathfrak{U} } ((r - T')  \phi_{n} ) = 0  .
\]
Therefore $ T' \phi_{ 0 } = r(T)  \phi_{ 0 } $ with $\phi_{0} \in \SA$.
\end{proof}
\begin{corollary}\label{cor:r-pole}
If $ r(T) $ is a pole of the resolvent, then its order is maximal among the poles in the peripheral spectrum 
$\{ \alpha \in \sigma(T) \colon \abs{ \alpha}= r(T) \} $ of $T$.
\end{corollary}
\begin{proof}
Let $r = r(T) $, $ \alpha $ in the peripheral spectrum of $T$ be a pole of the resolvent of order $ k $ and let $ \lambda = \mu \alpha $ where $ \mu > r$.
Then for $ 0 \leq \phi \in \BA^{*} $, $ 0 \leq x \in \BA $, both of norm $\leq 1$,  and $m > k$ we obtain by Equation~\eqref{eq:resolvent-inequality} on page~\pageref{eq:resolvent-inequality}
\[
	\abs*{ ( \lambda - \alpha )^{ m }\phi(( \RR( \lambda, T)x) }
		\leq
			(\mu - r )^{ m } \norm{ \RR( \mu, T) }  \, .
\]
Again by the uniform boundedness principle and the generating properties of the positive cones in $\BA$ and $\BA^{*}$ this shows that the pole order of $ r(T) $ is maximal among the poles in the peripheral spectrum.
\end{proof}

%% file: ams-sec3.tex

\section{Positive Semigroups}
This section establishes the integral representation of the resolvent of the generator $A$ of a positive $C_{0}$-semigroup on both \CA-algebras and the preduals of \WA-algebras.
We extend the standard integral representation, valid for $ \Re( \lambda ) > \omega_{0}$, to the larger domain $ \Re( \lambda) > \sA $.
The main result is a unified theorem with a single proof that covers both settings.
\begin{theorem}\label{thm:resolvent-eq}
Let $E$ be a \CA-algebra or the predual of a \WA-algebra and let $A$ be the generator of a positive $C_{0}$-semigroup $(T(t))_{t\geq0}$ on $E$.
Then for all\/ $ \lambda \in \C $, $ \Re(\lambda) > \sA$ and $ x \in E$ we have
\begin{equation}\label{eq:resolvent}
	\RR( \lambda,A )x = \lim_{t \to \infty} \int_{0}^{t} \eu^{- \lambda s} T(s)x \ds .
\end{equation}
\end{theorem}
\begin{proof}
For $ \lambda_{0} > \omega_{0}$ , $ \sA < \mu < \lambda_{0}$, $0 \leq x \in E$ and $ 0 \leq x' \in E'$.
As in the proof of \cite[C-III, Theorem 1.2]{nagel-ln:2025}, we obtain
\begin{equation}\label{eq:resolvent-eq}
	\langle \RR(\mu,A)x , x' \rangle =
		\lim_{t \to \infty} \langle \int_{0}^{t} \eu^{- \mu s} T(s)x \ds , x' \rangle .
\end{equation}
For $ t \geq 0 $ let $f_{t}$ and $f$ be defined on the $ \sigma(E',E)$-compact set 
\[
	K := \{ 0 \leq x' \colon x' \in E', \norm{x'} \leq 1 \}
\]
by
\[
	f_{t} \colon x' \mapsto  \langle \int_{0}^{t} \eu^{- \mu s} T(s)x \ds , x' \rangle 
	\quad \text{and} \quad
	 f  \colon x' \mapsto \langle \RR(\mu,A)x , x' \rangle .
\]
Then these functions are elements of $ C(K) $.
Thus the net $ (f_{t})_{t}$ is pointwise convergent to $f$ and monotone increasing, hence it is uniformly convergent in $C(K)$ by Dini's theorem. 
Since every $x' \in E'$ is a linear combination of at most four positive elements of $E'$, we have
\begin{equation}\label{eq:resolvent-strong}
	 \RR(\mu,A)x  =
		\lim_{t \to \infty}  \int_{0}^{t} \eu^{- \mu s} T(s)x \ds  .
\end{equation}
in the norm of $E$.

If $x \in E$ and $x' \in E'$, then there exist  positive $x_{j}  \in E$ (respectively   
positive $x'_{j} \in E'$) for $ j \in \{ 0, \ldots , 3 \}$ such that
\[
	x = \sum_{j=0}^{3} \iu^{j}x_{j}
	\quad \text{resp.} \quad
	x' = \sum_{j=0}^{3} \iu^{j}x'_{j}
\]
We define the positive elements $\abs{x}$ and $\abs{x'}$ by 
\[
	\abs{x} = \sum_{j=0}^{3} x_{j}
	\quad \text{resp.} \quad
	\abs{x'} = \sum_{j=0}^{3} x'_{j}
\]
and remark, that 
\[
	\abs{ \langle T(s)x , x' \rangle} \leq \langle T(s) \abs{x} , \abs{x'} \rangle  
\]
for $ s \geq 0 $.

Now we argue as in the second part of the proof of \cite[C-III, Theorem 1.2]{nagel-ln:2025}.
Take $ \lambda \in \C$ such that $ \Re(\lambda) > \sA$, then $ \lambda = \mu + \iu \gamma$ with $ \sA < \mu$ and $ \gamma \in \R$.

If $x \in E$, $ x' \in E'$ and $ 0 \leq r < t < \infty$, then 
\begin{equation}
	\begin{split}
		\abs*{ \left\langle \int_{r}^{t} \eu^{- \lambda s } T(s)x \ds , x' \right\rangle}
		&\leq
		\int_{r}^{t} \eu^{- \Re(\lambda) s } \abs{ \langle T(s)x \ds , x' \rangle } \\
		&\leq
		\int_{r}^{t} \eu^{- \mu s }   \langle T(s)\abs{x} \ds , \abs{x'}  \rangle .
	\end{split}
\end{equation}
But there exists $ m > 0$ such that $ \norm{\, \abs{ x'} \,} \leq m $ for every $ x' \in E' $ with $ \norm{x'} \leq 1$.
Therefore
\begin{equation}
	\begin{split}
		\norm*{  \int_{r}^{t} \eu^{- \lambda s } T(s)x \ds }
		&=
		\sup_{\norm{x'}\leq 1} \,  
		\abs*{ \left\langle \int_{r}^{t} \eu^{- \Re(\lambda) s } T(s)x \ds , x' \right\rangle } \\
		&\leq 
		m \, \norm*{ \int_{r}^{t} \eu^{- \mu s }  T(s)\abs{x} \ds } 
	\end{split}
\end{equation}
which shows that for every $x\in E$
\[
	\lim_{t \to \infty}  \int_{0}^{t} \eu^{- \lambda s} T(s)x \ds  
\]
exits and
\[
	\RR(\lambda,A)x = \int_{0}^{\infty} \eu^{- \lambda s} T(s)x \ds  
\]
by \cite[A-I, Prop. 1.11]{nagel-ln:2025}.
\end{proof}
The next corollary is an immediate consequence of Theorem~\ref{thm:resolvent-eq} and should be compared with Theorem~\ref{thm:main-op} on page~\pageref{thm:main-op}.
\begin{corollary}
Let $E$ be a \CA-algebra or the predual of a \WA-algebra and let $A$ be the generator of a positive $C_{0}$-semigroup $(T(t))_{t\geq0}$ on $E$.
For all $ 0 \leq x' \in E'$, $0 \leq x \in E$ and $ \Re(\lambda) > \sA$ we have
\begin{equation}\label{eq:resolvent-inequality-c0}
	\abs*{ \langle \RR( \lambda, A ) x , x' \rangle }
			\leq \langle \RR( \Re(\lambda),A)x , x' \rangle .
\end{equation}
\end{corollary}
Now, we are able to prove our main result.
\begin{theorem}\label{thm:main-sg}
Let $(T(t))_{t\geq0}$ be a $C_{0}$-semigroup of positive operators on a \CA-algebra $\BA$.
\begin{enumerate}

\item
Then $ \sA = \omega_{0}$.

\item
If\/ $\Sp(A) \neq \emptyset$, \ie $- \infty < \sA$, then $\sA \in \Sp(A)$.

\item 
If\/ $\BA$ has a unit $\1$, then $\Sp(A) \neq \emptyset$ and there exists a state $\phi \in \BA^{*}$ such that $A'\phi = \sA \phi$ for the adjoint operator of\/ $A$.
	
\end{enumerate}
\end{theorem}
\begin{proof}
{(Proof of \upshape(i)})\quad
By rescaling  we take $\omega_{0}=0$ and we assume $\sA < 0$.

If $\phi \in \BA^{*}$, then the function
\[
	t \mapsto \|T'(t)\phi\| = \sup_{\|x\| \leq 1} \abs{\phi(T(t)x)}
\]
is lower semicontinuous and hence measurable on $\R_{+}$, which justifies the integrals below.
Furthermore, $ T'(t)\phi \in D(A')$ and the semigroup  $\{T(t)'\}$ is strongly continuous on $D(A')$ (\cite[A-I, 3.4]{nagel-ln:2025}).

Next let $ 0 \leq \phi \in \BA^{*}$ and $ 0 \leq x \in \BA$.
Then
\begin{equation}
	\begin{split}
		\left\langle \int_{0}^{t} T(s)' \phi \ds , x \right\rangle 
		&=
		\phi ( \int_{0}^{t} T(s) x \ds ) \\
		\leq \phi ( \RR(0,A) x \ds )
		&=
		\langle \RR(0,A)' \phi , x \rangle 
	\end{split}
\end{equation}
hence
\begin{equation}\label{eq:tslessreolvent}
		\int_{0}^{t} T(s)' \phi \ds \leq \RR(0,A)' \phi .
\end{equation}
Since the norm of $\BA'$ is additive on the positive cone of $\BA$, we obtain from 
Equation~\eqref{eq:tslessreolvent}
\begin{equation}
 \begin{split}
 		\int_{0}^{t} \norm*{ T(s)' \phi \ds } 
 		&=
		\norm*{ \int_{0}^{t} T(s)' \phi \ds }  \\
 		\leq \norm{ \RR(0,A)' \phi } & \leq \norm{ \RR(0,A)} \norm{\phi} \leq m_{\phi}.
 \end{split}
\end{equation}
for some $ m_{\phi} > 0 $ independent of\/ $t \in \R_{+}$.

Now, every $\phi \in \BA^{*}$ is a linear combination of at most four positive linear forms.
Thus for some $\gamma > 0$
\[
	\int_{0}^{t}\|T(s)'\phi\| \ds \leq \gamma  
\]
for all $t \in \R_{+}$ and therefore
\[
	\int_{0}^{\infty}\|T(s)'\phi\| \ds \leq \gamma  .
\]
Now we argue as in the theorem of Datko-Pazy (\cite[Thm. V.1.8]{EN2000}) that
\[
	\lim_{t \to \infty} \|T(t)'\| = 0 .
\]
Because of $ \|T(t)'\| = \|T(t)\| $ this implies 
\[
	\lim_{t \to \infty} \|T(t)\| = 0 .
\]
Hence $\sA < 0$ implies $\omega_{0} < 0$ which implies $\sA = \omega_{0}$ by using the rescaling techniques described in \cite[A-I, 3.1]{nagel-ln:2025}.
\footnote{I have to thank W. Arendt for his valuable remarks concerning the proof of this part of the theorem.}

{(Proof of \upshape(ii))}\quad
We suppose $\Sp(A) \neq \emptyset$ (\ie $\sA > -\infty$) and assume $\sA \notin \Sp(A)$.
Then there exist $\epsilon > 0$ and $\alpha_{0}$, $\beta_{0} \in \R$ such that
\[
	\rointerval{\sA-\epsilon,\infty} \subset \Res(A), 
		\quad \mu_{0} \coloneqq \alpha_{0} + \iu \beta_{0} \in \Sp(A) 		
		\quad \text{and} \quad \alpha_{0} > \sA - \epsilon.
\]
Next take $ \mu_{0} $ as above, choose $ \alpha_{n} \downarrow \alpha_{0}$ and let $ \lambda_{n} = \alpha_{n} + \iu \beta_{0}$.
Then for all $ 0 \leq \phi \in \BA^{*}$ and $0 \leq x \in \BA$ of norm $\leq 1$  Equation~\eqref{eq:resolvent-inequality-c0} yields
\[
	\abs{ \phi (\RR(\lambda_{n},A)x) } \leq \norm{\RR( \alpha_{n}, A )} .
\]
By the generating properties of the positive cones of $\BA$ and $\BA^{*}$ and the uniform boundedness principle, there exist $0 \leq \phi \in \BA^{*}$ and $0 \leq x \in \BA$ such that the left side of this equation is unbounded and hence the right side as well. 
This contradicts our assumption and therefore $\sA \in \Sp(A)$.

{(Proof of \upshape(iii))}\quad
Since all operators $T(t)$ $(t > 0)$ are positive and our algebra has a unit, there exists by Theorem~\ref{thm:main-op} for each $t > 0$ a state $\phi_{t}$ such 
\[
	T(t)' \phi_{t} = r(T(t)) \phi_{t} = \eu^{\omega_{0} t} \phi_{t}  .
\]
For $n \in \N$ let
\[
	E_{n} := \{\phi \in \SA \colon T(2^{-n}){}' \phi = e^{\omega_{0} 2^{-n}} \phi\}.
\]
The sequence of\/ $ (E_{n}) $,  is monotonically decreasing and due to the 
$\sigma(\BA^{*}, \BA)$-compactness of the unit ball of\/ $\BA^{*}$ we have 
\[
	\phi \in \bigcap_{n \in \N} E_{n} 
	\quad \text{and} \quad
	\phi(\1) = 1.
\]
From the semigroup property $T(s + t) = T(s) T(t)$ $(s, t \in \R_+)$ and the $\sigma(\BA^{*}, \BA)$-continuity of the mapping $t \mapsto T(t)'$ from $\R_{+} $ into $\BA^{*}$, 
it follows that 
\[
	T(t)' \phi = \eu^{\omega_{0} t}  \quad (t > 0) .
\]
Suppose $\omega_{0} = -\infty$.
Then $T(t)'\phi = 0$, so $0 = \phi(T(t)\1)$ for all $t > 0$.
By continuity this means $1 = \phi(\1) = 0$, a contradiction that can only be resolved by $\omega_{0} > -\infty$.

Since $e^{\omega_{0} t}$ is an element of the point spectrum of each operator $T(t)'$, it is in the residual spectrum of $T(t)$ for all $(t > 0)$.
But then $\omega_{0} \in \Sp(A)$ which yields $\sA = \omega_{0}$ (see \cite[A-III, 6.5]{nagel-ln:2025}).
\end{proof}
\begin{corollary}
Suppose that $s(A)$ is a pole of the resolvent $\RR( \lambda , A)$.
Then its order is maximal for all other poles on the line $ \sA + \iu \R$.
\end{corollary}
\begin{proof}
We can assume $\sA=0$ by rescaling the semigroup.
Suppose the order of the pole $\sA$ is $k$ and let $ \lambda = \alpha + \iu \beta $ for $\beta \in \R$ and $\alpha > 0$.
If $0 \leq \phi \in \BA^{*}$ and $ 0 \leq x \in \BA$ of norm $ \leq 1$, then for $ m > k $
\[
	\abs{\phi ( \lambda^{m} \RR( \lambda , A )x ))} \leq 
		  \alpha^{m} \norm{ \RR( \alpha , A )} 
\]
and the claim follows as in Corollary~\ref{cor:r-pole}.
\end{proof}
\begin{corollary}
Let $A$ be the generator of a strongly continuous group of positive operators on a \CA-algebra.
Then $ \sigma(A) \neq \emptyset$ and therefore $ \sA \in \Sp(A)$ and $\sA = \omega_{0}$.
\end{corollary}
\begin{proof}
$A$ and $-A$ are generators of positive $C_{0}$-semigroups on $\BA$ and suppose $ \Sp(A) = \emptyset$.
Then $ \sA = \mathrm{s}(-A) = - \infty$ and the set
\[
	\{ \RR( \lambda, A ) \colon \Re \lambda \geq 0 \}
	\quad \text{and} \quad
	\{ \RR( \lambda, - A ) \colon \Re \lambda \geq 0 \}
\]
is bounded in $\L{\BA}$ using Equation~\eqref{eq:resolvent-inequality-c0}, \ie the set $\{ \RR( \lambda, A ) \colon \lambda \in \C\} $ is bounded.
But then the function $ \lambda \to \RR( \lambda , A ) $ is constant by Liouville's theorem, hence identically zero because $ \lim_{ \lambda \to \infty} \RR( \lambda, A) = 0$.
But this is not possible, hence $ \Sp(A) \neq \emptyset$ and the claim follows.
\end{proof}

For the final proposition we use the fact that every \WA-algebra has a unit (\cite[Prop. 1.6.1]{sakai:1971}).
\begin{proposition}

\begin{enumerate}
\item
If $\BA$ is a \WA-algebra with predual $\BA_*$ and $ (T(t))_{t\geq0} $  a positive $C_{0}$-semigroup on $\BA_*$, then $\sA = \omega_{0}$.

\item
If $\BA$ is a \CA-algebra and  $ (T(t))_{t\geq0}$ a positive $C_{0}$-semigroup on the dual\/ $\BA^{*}$ of\/ $\BA$, then $\sA = \omega_{0}$.
\end{enumerate}

\end{proposition}
\begin{proof}
\begin{enumerate}[\upshape (i), wide, labelindent=0.0em]
\item
By rescaling we may assume $ \omega_{0} = 0 $ and suppose $ \sA < 0 $. 
Then $ \RR(0,A)$ is a positive operator and by Theorem~\ref{thm:resolvent-eq} given by
\[
	R(0,A) \phi = \int_{0}^{\infty} T(s) \phi \ds .
\]
If $ \phi $ is positive, then 
\[
	\norm{R(0,A) \phi} = \langle R(0,A) \phi, \1 \rangle
		= \langle \int_{0}^{\infty} T(s) \phi \ds , \1 \rangle
		= \int_{0}^{\infty} \norm{T(s) \phi} \ds .
\]
Again, this implies $ \int_{0}^{\infty} \norm{T(s) \phi} \ds < \infty $ for every $\phi \in \BA_{*}$.
From the Theorem of Datko-Pazy we obtain 
\[
	\lim_{t \to \infty} \|T(t)\| = 0 
\]
hence $\omega_{0} < 0$ which implies $\sA = \omega_{0}$.

\item
Since the bidual of a \CA-algebra is a \WA-algebra, the claim follows immediately.
\end{enumerate}

\end{proof}
\begin{remarks}\label{ex:remarks}
\begin{enumerate}

\item
There exist non-positive $C_{0}$-semigroups on \CA-algebras with $\sA < \omega_{0}$.
For this see \cite{zbmath03766522}.

\item
For examples of positive $C_{0}$-semigroups  on a (commutative) \CA-algebra without unit and $ -\infty = \sA = \omega_{0}$ or $ -\infty = \sA < \omega_{0} $ see \cite[B-III, Examples 1.2.(a) \& (b)]{nagel-ln:2025}.

\item
\cite[B-III, Examples 1.7]{nagel-ln:2025} shows, that a unit in $\BA$ is necessary for the existence of an eigenvector to the eigenvalue $\sA$ for $A'$.

\end{enumerate}
\end{remarks}